\documentclass[12pt]{article}
\usepackage{amsfonts}

\newcommand{\lab}[1]{\label{#1}}                

\newcommand{\bel}[1]{\be\lab{#1}}
\newcommand\eqn[1]{(\ref{#1})}

\newcommand{\be}{\begin{equation}}
\newcommand{\ee}{\end{equation}}
\newcommand{\bea}{\begin{eqnarray}}
\newcommand{\eea}{\end{eqnarray}}
\newcommand{\bean}{\begin{eqnarray*}}
\newcommand{\eean}{\end{eqnarray*}}
\newtheorem{thm}{Theorem}

\newtheorem{con}{Conjecture}
\newtheorem{lemma}{Lemma}
\def\proof{\noindent{\bf Proof}\ \ }
\def\qed{~~\vrule height8pt width4pt depth0pt}

\newcommand{\ZZ}{\ensuremath{\mathbb{Z}}}
\def\ex{{\bf E}}

\def\T{{\cal T}}
\def\HH{{\cal H}}
\def\R{{\cal R}}

\def\yv{{\bf y}}
\def\zv{{\bf z}}
\def\rv{{\bf r}}
\def\vv{{\bf v}}

\def\eps{\epsilon}

\def\th{\theta}

\def\non{\nonumber}
\def\no{\noindent}

\catcode`@=11
\@addtoreset{equation}{section}

\catcode`@=12

\begin{document}

\title {Tournaments with many Hamilton cycles}
\author{N. C. Wormald 
\thanks{Research supported by the Canadian
Research Chairs Program and NSERC}
\\ Department of Combinatorics and Optimization\\ University of
Waterloo\\ Waterloo ON\\ Canada N2L 3G1\\ {\small
nwormald@uwaterloo.ca}\\}
\date{}
\maketitle
\begin{abstract}

The object of interest is the maximum number, $h(n)$, of
Hamilton cycles in an $n$-tournament. By considering the expected number of
Hamilton cycles in various classes of random tournaments, we obtain new asymptotic
lower bounds on $h(n)$. The best result so far is approximately $ 2.85584...$ times
the expected number $g(n)$ of Hamilton cycles in a random $n$-tournament, and it is
conjectured that $h(n)\sim cg(n)$ where $c\approx 2.855958$.     The same
statements hold for Hamilton paths.  
\end{abstract}
 
\no {\small AMS Classifications: 05A16, 05C30, 05C80, 05C38, 05D40}
\section{Results}
A tournament is an orientation of a complete graph. A {\em Hamilton} cycle or
{\em path} is a tournament $T$ is a directed cycle or path which contains all
vertices in $T$. Let
$H(T)$ ($H'(T)$) denote the number of Hamilton cycles (paths) in a tournament $T$,
and let
$S_n$ denote the set of all tournaments on
$n$ vertices. Define
$$h(n) =  \max_{T\in S_n} \   H(T),\quad h'(n) =  \max_{T\in S_n} \  
H'(T).$$

Let $g(n)=2^{-n}(n-1)!$ and $g'(n)=2^{-n+1}n!$. If $T$ is chosen
uniformly at random from
$S_n$, it is easy to calculate that the expected value $\ex[ H]$ of $H$ is $g(n)$
and similarly that
$\ex [H']= g'(n)$, by which method Szele~\cite{S} originally showed that $h'(n)\ge 
g'(n)$. Adler et al.~\cite{AAR}  showed that
$h'(n)\ge (e-o(1))g'(n)$ as $n\to\infty$, the first improvement   over Szele's
original result. The proof examined the expected value of
$f(T)$ when $T$ is randomly drawn from a certain subset of the regular
tournaments (i.e.\ tournaments with equal indegree and outdegree at every vertex),
obtained by partitioning the arc set into triangles and randomly orienting each as a
directed 3-cycle. The best known upper bound is of the order of
$n^{c}g'(n)$ for $c=3/2$ by  Alon~\cite{A} using Br{\'e}gman's theorem
(formerly Minc's conjecture) on permanents. This was recently improved to
$c$  slightly less than $5/4$ by Friedgut and Kahn~\cite{FK} based on
Radhakrishnan's proof of Br{\'e}gman's theorem. The Hamilton cycle version
of both of these results also holds, with $h$ and $g$ in place of $h'$ and
$g'$.

The object of the present paper is to investigate various classes of tournaments to
 find clues on what affects the number of Hamilton cycles in a tournament.   As it
turns out, the expected number of Hamilton cycles in a random regular
tournament, which can be calculated using the methods of McKay~\cite{M}
combined with those in~\cite{MW}, coincides with the lower bound
$(e+o(1))g(n)$ on $h(n)$ obtained in~\cite{AAR}. We
find several classes of random tournaments which
improve this result. Most of these, but not all, are regular. 

The discussion in~\cite{FK} also asked for
the minimum value of
$H(T)$ when
$T\in S_n$ is regular, a question which we make little progress on here. However,
we do find some classes of random regular tournaments with significantly fewer
Hamilton cycles than the expected number.

We obtain analogous results for Hamilton paths, but prefer to present the cycle
results first.  In~\cite{AAR} it is noted as interesting to determine if
$h'(n)=\Theta(n!/2^n)$. We make no conclusive progress on this, but  
the results here are suggestive enough for us to conjecture the affirmative and to
estimate $h'(n)$. 

Let $d\ge 1$ be odd and consider a fixed tournament $T^*$ with $d$ vertices. Let
$\T(T^*)=\T(T^*,n)$ denote the set of random tournaments on $dn$ vertices
constructed by choosing  $d$ tournaments
$T_1,\ldots, T_d$ independently uniformly at random from
$S_n$,  expanding each vertex
$i$ of
$T^*$ to the tournament $T_i$, and expanding each arc $ij$ of $T^*$ to all arcs
from vertices of $T_i$ to vertices of $T_j$. (To keep vertex labels
distinct, renumber the vertices of the resulting tournament in a
canonical way.)

Also define  $\R(T^*)=\R(T^*,n)$ to be the restricted probability space formed from 
$\T(T^*)$, by conditioning on the event that each $T_i$ is a regular
tournament. (Assume that  $n$ is odd for this definition.) Define
 $\R_0(T^*)$ similarly, with the restriction to the set of tournaments
which can be obtained by orienting the triangles in a decomposition of
the edges of the complete graph $K_n$ into triangles. These were the
tournaments considered in the main argument in~\cite{AAR}. For all our asymptotic
statements, $n\to\infty$.

\begin{thm}\lab{t:timese}
Let $T^*$ be any fixed regular tournament with $d\ge 1$ vertices. Then 
$$\ex [H(\R(T^*))]\sim \ex [H(\R_0(T^*))]\sim e^{1/d} \ex[ H(\T(T^*))].$$
\end{thm}

The proof of this theorem, along with those of the other results stated here, are
 in the next section. 

The rest of the results concern $\ex [H(\T(T^*))]$, for
use in conjunction with Theorem~\ref{t:timese}. The following  case is given
separately because it does not require much computation. If $C_3$ denotes the
directed cycle on three vertices, then
$\T(C_3)$ consists of three random $n$-tournaments arranged in a 3-cycle with all
arcs between the three tournaments having consistent cyclic orientation.

\begin{thm}\lab{t:T3}
We have 
$$\ex [H(\T(C_3))]\sim 2g(n).$$
\end{thm}

For the next results, which concern
$\ex [H(\T(T^*))]$ for arbitrary regular tournaments $T^*$, we need some
definitions. Given a regular tournament $T^*$ on $d\ge 3$ vertices, let $A$ be the
adjacency matrix of
$T^*$, i.e.\  $A=(a_{ij})$ where $a_{ij}=1$ if $ij$ is a directed edge of $T$
and $a_{ij}=0$ otherwise. Put $D={\rm diag}(x_1,
\ldots, x_d)$, and define 
$f(x_1,\ldots ,x_n)$ to be the $(1,1)$ entry of $(I_d-AD)^{-1}$ and
 $\hat f (x_1,\ldots ,x_n)=\frac{1}{d}{\rm trace}((I_d-AD)^{-1})$. Put
\bea 
f_n=f_n(T^*)&=&[y_1^n\cdots y_d^n]f\left(2y_1/(2-y_1),\ldots
,2y_d/(2-y_d)\right), \lab{fndef}
\\
\hat f_n&=&[y_1^n\cdots y_d^n]\hat f\left(2y_1/(2-y_1),\ldots
,2y_d/(2-y_d)\right),\non
\eea  
where square brackets denote the
extraction of coefficients.

\begin{thm}\lab{t:forT} 
Let $T^*$ be any fixed regular tournament with at least three vertices. Then 
$f_n\sim\hat f_n$ and
$$
\ex [H(\T(T^*))]\sim f_ng(dn)\left( \frac{2}{d}\right)^{dn}(2\pi
n)^{(d-1)/2}\frac{d^{3/2}}{d-1}.
$$
\end{thm}
The next result is an asymptotic evaluation of $f_n$, for which we need a
bit more notation. Define $B$ to be the matrix whose $(i,j)$ entry
$B_{ij}$ is $-y_j$ if $A_{ij}=1$ and 0 otherwise, except for diagonal entries
$B_{ii}=1-\frac12 y_i$. That is, $B= I_d-(\frac12 I_d +A)D_y$
where $D_y=\mbox{diag}(y_1,\ldots, y_d)$.
  Let $K=K({\bf y})= \det B$. Set 
\bel{z}
z=2/d\quad\mbox{and}\quad{\bf z}=(z,z,\ldots ,z).
\ee
 Writing $K_j$ for the partial
derivative of $K$ with respect to $y_j$, it will be shown that $K_d({\bf z})\ne 0$ and
$K({\bf z})=0$, so (e.g.\ by observing that $K$ is linear in $y_d$) one may
solve
$K({\bf y})=0$ in the neighbourhood of ${\bf z}$ and obtain
$y_d=h(y_1,\ldots, y_{d-1})$ for an analytic function $h$. Define a
function $\tilde f$ in the neighbourhood of $(0,0,\ldots,0)$ (the
$(d-1)$-dimensional vector) by
$$\tilde f(\theta_1,\ldots , \th_{d-1}) = \log h(ze^{i\theta_1},\ldots,
ze^{i\theta_{d-1}}).
$$
 Let
$\HH$ denote the determinant of the Hessian (matrix of second order
partials) of $\tilde f$ at
$\th_j=0$ for  $j=1, \ldots, d-1$.
\begin{thm}\lab{t:coeff} 
 Provided $\HH\ne 0$,  
$$
\ex [H(\T(T^*))]  \sim    
g(dn)\sqrt\frac{d}{ \HH}.
$$
\end{thm}

\no
{\bf Notes  \ } It will be shown in the proof that the same result holds
if we make definitions based on $K_j$ for any $1\le j\le d$. Also,
as will be seen in the proof, $K_j(z)$ is equal to $-d/2$ times the
$(j,j)$ cofactor in $\det B({\bf z})$.  Perhaps the condition
$\HH\ne 0$ can be shown to be true for all regular tournaments $T^*$, but
its value needs to be computed in any case.
\smallskip

The required
calculations can be done individually for small tournaments
$T^*$. Let $R_d$ denote the cyclic regular tournament of order $d$,
with arcs from $i$ to $i+r$ (mod $d$) for $1\le r<d/2$.  For $d=9$ only, let
$R'_9$ be the regular 9-vertex tournament in $\R(R_3,3)$ where each of
$T_i$ is a copy of
$R_3$. 

The values of $\sqrt{d/\HH}$ in the following table were obtained by
the algebraic manipulation package Maple. (Presumably there is a more direct way to
calculate this. In particular, it is unexplained why
$d/\HH$ turns out to be a perfect square in all the examples computed.)  For
large $d$, the exact ratio is omitted to save space. The next two rows are just
to make comparisons easy. The second line is 
$\rho_\T(T^*):=\lim_{n\to\infty}  \ex [H(\T(T^*))]/g(dn)$, which (by
Theorem~\ref{t:coeff}) is the same quantity as in the first line, rounded to four
decimal places for
$d\le 13$. Notice that these results refer to the non-regular tournaments in
$\T(T^*)$. For regular tournaments, the bottom row gives
$\rho_\R(T^*):=\lim_{n\to\infty} 
\ex [H(\R(T^*))]/g(dn)$, which comes by Theorem~\ref{t:timese} on multiplying by
$e^{1/d}$. 

\begin{center}
\begin{tabular}{|c|c|c|c|c|c|c|c|c|}	\hline
 $T^*$\raisebox{-0.3cm} {\rule{0cm}{0.8cm}}
 & $R_3$  & $R_5$ &$ R_7 $ &$R_9$ &$R_9'$  
&$R_{11}$ &$R_{13}$
\\ \hline \hline
 $\sqrt{d/\HH}$ \raisebox{-0.3cm} {\rule{0cm}{0.8cm}} &
$2$ &
$\frac{44}{19}$ & $\frac{3032}{1231}$
&$\frac{ 437584}{171769}$&$ \frac{ 5488}{2197}$
&$\frac{108859552}{41833969}$&
$\frac{41637757888}{15769379963}$
\\ \hline
$\rho_\T(T^*)$
\raisebox{-0.3cm} {\rule{0cm}{0.8cm}}
 & 2.0000 & 2.3158 & 2.4630  & 2.5475 & 2.4980 
& 2.6022 &2.6404 
\\ \hline
$\rho_\R(T^*)$ \raisebox{-0.3cm}
{\rule{0cm}{0.8cm}}
&
 2.7912 & 2.8285 & 2.8413 &
2.8469 & 2.7915
& 2.8498  & 2.8515 
\\ \hline
\end{tabular}
\end{center}

\begin{center}
\begin{tabular}{|c|c|c|c|c| }	\hline
 $T^*$\raisebox{-0.3cm} {\rule{0cm}{0.8cm}}
 & $R_{21}$  & $R_{41}$  & $R_{81}$  & $R_{101}$
\\ \hline \hline

$\rho_\T(T^*)$
\raisebox{-0.3cm} {\rule{0cm}{0.8cm}}
 &2.7215188\ldots   & 2.7867038\ldots & 2.820801899\ldots
& 2.827747088 \ldots
\\ \hline
$\rho_\R(T^*)$ \raisebox{-0.3cm}
{\rule{0cm}{0.8cm}}
&2.8542501\ldots& 
 2.8555079\ldots &
2.855842467\ldots&
2.855883644\ldots \\ \hline
\end{tabular}
\end{center}
(To avoid using prohibitively large amounts of memory when calculating with
Maple  for large $d$, the $(j,k)$ entry of the Hessian was computed separately
for each $j$ and $k$, first setting $y_i=2/d$ in the matrix $B$ if
$i\notin \{j,k,d\}$.) 

All other regular tournaments
$T^*$ which were examined in this way produced lower numerical values. The
 7-vertex tournament  with cyclic triples $abc$ and $a'b'c'$, with
arcs $aa'$, $bb'$, $cc'$, $ad$, $bd$ and $cd$, and other arcs making
it regular, gave  
$\sqrt{d/\HH}=\frac{3088}{1287}\approx 2.3994$.
For  9-vertex tournaments, the circulant tournament with arcs
$ij$ iff
$j-i (\mbox{mod } d)\in D$ yields  values for $\ex [H(\T(T^*))]/g(dn)$ of
$1.4582$ ($D=\{1,3,4,5 
\}$),  1.6811 ($D=\{1,2,4,5  \}$), 2.4557 ($D=\{1,2,3,5  \}$), 1.5618
($D=\{2,3,4,5  \}$), 1.5529 ($D=\{1,3,5,6  \}$), 1.9325 ($D=\{1,4,5,6 
\}$), 1.8117 ($D=\{ 2,3, 5,6  \}$).  However,  this
examination was not exhaustive in any way.

We conclude from the $R_{101}$ result that
$h(n)>  2.85588g(n)$ for infinitely many $n$. Incidentally we obtained the
asymptotic upper bound $1.4582e^{1/9}g(n)\approx 1.63g(n)$ on the minimum number of
Hamilton cycles in a regular $n$-tournament, but make no further efforts in
this direction. We
 turn instead to Hamilton paths. 
\begin{thm}\lab{t:paths}
Theorems~\ref{t:timese} and~\ref{t:coeff}  are both valid with $H$
and $g$ replaced by $H'$ and $g'$.
\end{thm}

Thus the table above is valid when interpreted for Hamilton paths rather than
Hamilton cycles. It follows that  
$h'(n)>  2.85588g'(n)$ for infinitely many $n$.

\bigskip

\no
{\bf   Concluding Remarks and Conjectures\ }
\smallskip

In this paper we do not consider the variance of $H$ or $H'$. Routine calculations
show that if $T$ is chosen uniformly at random from
$S_n$, then the variance of $H$ is $o(g(n)^2)$. Thus, by Chebyshev's
inequality, a random tournament with probability close to 1 contains approximately
$g(n)$ Hamilton cycles. A significantly larger variance would have implied the
existence of tournaments with significantly more than $g(n)$ Hamilton cycles.
Presumably a similar concentration holds for the other models of random tournament
considered in this paper.

In~\cite[Proposition 2.5]{A} it was shown that for every tournament $Q$
there exists a tournament with exactly one more vertex than $Q$ (and containing
$Q$) and with at least
$H'(Q)/4
$ Hamilton cycles. Thus $h'(n)>  cg'(n)$ implies
that $h(n)>  cg(n)$, and therefore the conclusions of this type in this paper can be
obtained by looking at $H'$ alone.  However, this argument does not quite permit us
to   derive $\ex [H(\T(T^*))]$ given $\ex [H'(\T(T^*))]$, and it is
of interest to know about both of these quantities. In any case, for the
results in this paper the cycle case is a little simpler to consider than
the path case.

It is interesting to speculate why $R_9$ has more Hamilton
cycles than
$R'_9$: perhaps something to do with the decreased number of 3-cycles
expected to be spanned by the out-set of each vertex.  This suggests that
perhaps $T^*=R_d$ maximises $\ex [H(\R(T^*))]$ for $|V(T^*)|=d$.
 Alon~\cite{A} felt it plausible that the cyclic tournament $R_n$
maximises the number of Hamilton cycles of any $n$-tournament, that is, 
$H(R_n)=h(n)$. The results reported in and after the table suggest this quite
strongly, so we have turned this into a conjecture (see below). The tournaments
in
$H(\T(R_d))$  possess a large proportion (approximately $(d-1)/d$) of
the arcs of
$R_{dn}$, with the other arcs random. $H(\T(R_d))$ is a regular version
of this.

Both~\cite{AAR} and~\cite{FK} consider that $h(n)=\Theta (g(n))$ is possible, which
would imply the existence of
$\liminf_{d\to\infty}\max_{|V(T^*)|=d}
\ex [H(\R(T^*))]/g(dn)$.
We can
estimate this $\liminf$ to be approximately $2.8559579$ in the following way, 
assuming that
$R_d$ achieves the maximum for each $d$. Firstly, fitting the polynomial
$c_0+c_1/d+c_2/d^2+c_3/d^3+c_4/d^4$ in
$1/d$ to
$\rho_\R(R_d)$ for $d=11$, 21, 41, 81 and 101 gives $c_0=2.85595787\ldots$.
Secondly, the value of $c_1$ obtained is less than $10^{-4}$. This
suggests that  perhaps 
$c_1=0$. (Also
$c_3$ appears to be small by the same criterion, but not $c_5$.) This is
supported by the fact that fitting
$c_0+c_2/d^2+c_3/d^3+c_4/d^4$ to
$\rho_\R(R_d)$ for
$d=11$, 21, 41 and 81 only, gives a polynomial $P(1/d)$ for which
$\rho_\R(R_{101})- P(1/101)<2\times 10^{-9}$. Fitting   
$c_0+c_2/d^2+c_3/d^3+c_4/d^4+c_5/d^5$  to
$\rho_\R(R_d)$ for $d=11$, 21, 41, 81 and 101 gives $c_0=2.85595789\ldots$, in
agreement with the estimate above. 

One might also attempt to extrapolate the exact values of
$\sqrt{d/\HH} =\rho_\T(R_{d})$ from the values for small $d$. However, this
may not be easy, as indicated by the prime factorisation of the denominator in
the case
$d=11$, as $53\times 297535471$.

The results above lead to the following conjectures, some of which are 
obviously stronger than the assertion that $h(n) =\Theta(g(n))$.  
\begin{con}  \lab{con1}
 For all $d\ge 3$, the maximum of 
$  \ex [H(\T(T^*))] $, over all tournaments $T^*$ on $d$ vertices ($d$ odd,
and $n$ sufficiently large), is achieved for $T^*=R_d$.
\end{con}
\begin{con}  \lab{con2}
  The limit $\lim_{d\to\infty} \rho_\R(R_d)$ ($d$ odd) exists.
\end{con}  
\begin{con}  \lab{con3}
  $H(R_n)=h(n)$ for $n$ sufficiently large (and odd).
\end{con}  
\begin{con}  \lab{con4}
  The limit
$\lim_{n\to\infty} h(n)/g(n)$  exists.
\end{con}
 Assuming that Conjectures~\ref{con2} and~\ref{con4} are true, we believe 
that they refer to the same constant as estimated above, as follows.
\begin{con}   \lab{con5}
 The limits $\lim_{d\to\infty} \rho_\R(R_d)$  and $\lim_{n\to\infty}
h(n)/g(n)$ are equal and have the approximate value 2.855958.
\end{con}  
In the absence of verification of Conjectures~\ref{con1} and~\ref{con3}, 
the limits $ \lim_{d\to\infty}\max_{|V(T^*)|=d} \rho_\R(T^*)$ and
$ \lim_{n\to\infty}  H(R_n)/g(n)$ ($n$ odd)  are also of  interest, and we
conjecture that these limits exist and have the value in Conjecture~\ref{con5}.

 We also make the analogous  conjectures for  $H'$, $h'$ and $g'$.
Of independent interest we have the following.
\begin{con} For every regular tournament $T^*$ on at least three vertices,
$\HH\ne 0$ and $d/\HH$ is the square of a rational.
\end{con}

Finally, we note that the two parts of Theorem~\ref{t:timese}, referring $\R_0$ to
and to $\R$, have the same implications for $h(n)$, but we include both cases
since it may be significant that they do have the same implications.  

 \bigskip

\no
{\bf Acknowledgements\ }  Encouraging comments and suggestions on this problem were 
made by Jeff Kahn, Tomasz {\L}uczak and Svante Janson. 

\section{Proofs}
We first derive two formulae for $\ex [H(\T(T^*))]$.

\begin{lemma} \lab{formulae}
Let $T^*$ be any fixed  tournament with at least three vertices. Then
\bea
\ex [H(\T(T^*))] &=& \non\\
 &&\hspace{-2.5cm} \left(\frac{n!}{2^{n}}\right)^d\sum_{j_1,\ldots , j_d\ge
1} j_1^{-1}[x_1^{j_1}\cdots x_d^{j_d}] f(x_1,\ldots ,x_d)
\prod_{i=1}^d  \left( n-1\atop j_i-1\right) 2^{j_i}.
\lab{first}
\eea
\end{lemma}
\proof
 A
Hamilton cycle in
$Q$ projects onto a closed walk $W$ in $T^*$ (by mapping $T_i$ to the vertex
$i$ and ignoring loops). Let
$j_i$  denote the number of arcs of
$W$ whose head is 
$i$ (which is of course equal to the number with tail  $i$).  
Distinguish a special ``initial" vertex $v$ in $T_1$ such that the arc
$wv$ of $Q$ has $w\notin V(T_1)$. This introduces a multiplicity of $j_1$
into the counting. For each $i$, the
arcs of  $Q$ induce an ordered spanning forest $F_i$  of $T_i$, i.e.\ a
set of
$j_i$ disjoint directed paths (perhaps some containing only one vertex)
which cover the vertices of
$T_i$, the ordering determined by the initial vertex $v$. The number of
such spanning forests which can ever appear is $n!\left(n-1\atop
j_i-1\right)$, since they can be formed from a Hamilton path on $V(T_i)$
by deleting
$j_i-1$ edges. Any such spanning forest $F_i$ contains
$n-j_i$ arcs, so the probability that
$T_i$ contains $F_i$ is
$2^{-n+j_i}$. Moreover since the
$T_i$ are independent, the probability that they   contain a given set of
forests $F_1,\ldots, F_d$ is
$\prod_i 2^{-n+j_i}$. Given these forests and, in addition, the walk
$W$, we can reconstruct the original Hamilton cycle in 
$Q$. The number of closed walks
$W$   starting
at vertex 1, and with parameters $j_i$ as above, is
$[x_1^{j_1}\cdots x_d^{j_d}] f(x_1,\ldots ,x_d)$,  with $f(x_1,\ldots,
x_d)$ defined as before. The  lemma now follows by
linearity of expectation (and dividing by $j_1$ to ignore the choice of
$v$). \qed

\begin{lemma}\lab{final}
If $T^*$ is regular then, in the summation in~\eqn{first}, $(1-o(1))$ of the
contribution comes from terms with $j_i=(1+o(1))(d-1)n/d$ for all $i$.
In particular,
$$
\ex [H(\T(T^*))] \sim  \frac{d (n!2^{-n})^d}{(d-1)n}\sum_{j_1,\ldots , j_d\ge
1}  [x_1^{j_1}\cdots x_d^{j_d}] f(x_1,\ldots ,x_d)
\prod_{i=1}^d  \left(n-1\atop j_i-1\right) 2^{j_i}.
$$
\end{lemma}
\proof
 Consider restricting the summation in~\eqn{first} to
those parameters such that $\sum_i j_i = dj$ for some fixed $j$.
Since the out-degree of each vertex in
$T^*$ is $(d-1)/2$, there are $((d-1)/2)^{dj}$ walks of length $dj$, and
moreover  the position of a random walker will
converge to the uniform distribution. Thus, the sum of all the
coefficients of
$f$ in~\eqn{first} with $\sum_i j_i = dj$ is asymptotic to
$((d-1)/2)^{dj}/d$, and almost all of the value is contributed by the
terms with $|j_i-j|<\eps j/2d$ ($1\le j \le d$) for any fixed $\eps>0$.   
By making use of the
log-concavity of the factors in the product in~\eqn{first}, it is
straightforward to show that for the above-mentioned terms, this product is at
least as large as it is if $|j_i-j|>\eps j$ for any $i$. Hence terms of the
latter type contribute negligibly to the summation. For the other
(non-negligible) terms, Stirling's formula shows that the product has value
\bel{factor}
\frac{n^{dn}2^{dj} \exp(\Theta(\eps^2 n))}{j^{dj}(n-j)^{d(n-j)}}.
\ee
 
On the other hand, the product of this
with the sum of coefficients, $((d-1)/2)^{dj}/d$,  is easily seen to be
maximised for $j\sim (d-1)n/d$ (apart from the error term), and it is easy
to see that unless $|j-(d-1)n/d|=O(\eps)$ the
contribution will be negligible.  The
lemma follows (on taking $\eps \to 0$ slowly). \qed

Also for the proof of Theorem~\ref{t:timese} we need to know about the
expected number of spanning   forests of paths in random regular
tournaments. For a digraph  $D$, let $E(D)$ be its arc set,
$a(D)=|E(D)|$ its number of arcs, and $s(D)$ the number of path components
of length at least 1. Of course, here
$n$ is restricted to odd integers. The proof of the following theorem uses an
integral formulation of the number of regular $n$-tournaments which was used by
McKay~\cite{M} to count them asymptotically. For another application of
this formulation to a somewhat related problem, see McKay and
Robinson~\cite{MR}.
\begin{lemma}\lab{paths}
For each $n\ge 1$ let $F_n$ be either a directed graph on $n$ labelled
vertices with each component a directed path, or a directed Hamilton
cycle.  Then the probability that
$F_n$ appears in a random 
$n$-vertex regular tournament is asymptotically equal to
$2^{-a(F_n)}e^{a(F_n)/n-s(F_n)/n}$ as $n\to\infty$. 
\end{lemma}
\proof
Following~\cite{M}, the number of regular tournaments on
$n$ vertices is precisely
\bel{countx}
 RT(n)={1\over (2\pi i)^n}\oint\cdots\oint
  { \prod_{1\le j<k\le n}(x_j^{-1}x_k+x_jx_k^{-1})\over
x_1x_2\cdots x_n  }\,
  dx_1dx_2\cdots dx_n,
\ee
where each integral is around the circle $|z|=1$ in the anticlockwise
direction. The numerator of the integrand is the generating function for
tournaments enumerated by the excess of out-degree over in-degree at each
vertex. Under the substitution $x_j=e^{i\theta_j}$ this becomes
\bel{countth}
 RT(n)={2^{n(n-1)/2}\over (2\pi i)^n}\int_{U_n(\pi)}
   \prod_{1\le j<k\le n}\cos(\theta_j-\theta_k)\,
  d\theta_1d\theta_2\cdots d\theta_n 
\ee
where $U_n(x)$ denotes $[-x,x]^n$.
For the general situation it suffices to consider the
case that $F_n$ is a subdigraph of the directed path $1,2,\ldots ,n$, with
the ordering imposed on the paths components being the natural one. (The
Hamilton cycle case requires only a trivial modification to the argument
for the case $a(F)=n-1$, and we make no further remark on it.) Then  all
arcs of $F_n$ are of the form
$(j,j+1)$. Let $E=\{j:(j,j+1)\in E(F)\}$.  Then the number $N(F_n)$ of
tournaments containing $F_n$ is expressed as an integral
like~\eqn{countx}, but with the following factors altered:   if
$j\in E$ and $k=j+1$, the factor $(x_j^{-1}x_k+x_jx_k^{-1})$ is replaced
by $x_jx_k^{-1}$. This means that the expression in~\eqn{countth} for the
current problem becomes
\bel{adjustth}
 N(F_n) = {2^{n(n-1)/2-a(F_n)}\over (2\pi i)^n}\int_{U_n(\pi)}
  \frac{ \prod_{1\le j<k\le n}\cos(\theta_j-\theta_k)}{ \prod_{j\in
E}\cos(\th_j-\th_{j+1})}
\,\frac{\prod_{j\in E^+} e^{i\th_j}}{ \prod_{j\in E^-}  e^{i\th_j}}\,
  d\theta_1d\theta_2\cdots d\theta_n
\ee
where $E^+$ ($E^-$) is the set of vertices which begin (end) a directed
path. Thus $E^+=\{j:(j,j+1)\in E(F), (j-1,j)\notin E(F)\}$, and for
defining
$E^-$, switch $\in$ and $\notin$.

We continue to follow the argument in~\cite{M} and split the integral
in~\eqn{adjustth} into three
major parts. For some $\eps>0$, let $I_1$ denote the contribution
coming from those points at which
$|\th_j-\th_n|\le n^{-1/2+\eps/4}$ or $|\th_j-\th_n+\pi|\le
n^{-1/2+\eps/4}$ for $1\le j \le n-1$, where $\th_j$ values are mod
$2\pi$. Still valid in the present context is the argument in~\cite{M}
that we may set $\th_n=0$ and assume all the $\th_j$ are clustered near
0, which in this case gives
\bel{eye1}
I_1 = 2^n\pi
 \int_{U_{n-1}(n^{-1/2+\eps/4})}
  \frac{ \prod_{1\le j<k\le n}\cos(\theta_j-\theta_k)}{ \prod_{j\in
E}\cos(\th_j-\th_{j+1})}\,
\frac{\prod_{j\in E^+} e^{i\th_j}}{ \prod_{j\in E^-}  e^{i\th_j}}\,
  d \theta_1d\theta_2\cdots d\theta_{n-1}.
\ee
Expanding (most of) the integrand gives (dropping terms dominated by the
$o(1)$ error term),
\bea&&\exp\left(o(1)+\sum_{1\le j<k\le n} -\frac12
(\theta_j-\theta_k)^2-\frac{1}{12}(\theta_j-\theta_k)^4\right)\lab{line1}\\
&&\times 
\exp\left(\sum_{j\in E^+}  {i\theta_j}\, +
\sum_{j\in E^-}  {-i\theta_j}
 \, +
\sum_{j\in E}  \frac12 (\theta_j-\theta_{j+1})^2  
\right),\lab{line2}
\eea
where the first line contains the terms in the original
calculation in ~\cite{M}.

The next step is to apply the linear transformation 
$$y_j=\th_j-\sum_{k=1}^{n-1}\th_k/(n+\sqrt n),$$
for $1\le j\le n-1$, which implies using~\cite[Equation (2)]{M} that
$\th_j=y_j+\phi$ where 
$$\phi = \sum_{k=1}^{n-1} y_k/(1+\sqrt n).$$
The transformed logarithm of~\eqn{line1}  is dealt with in~\cite{M},
the significant terms being 
~\bel{old}
-\frac{n}{2} \sum_{k=1}^{n-1} y_k^2
-\frac{n}{12} \sum_{k=1}^{n-1} y_k^4-\frac{n}{4} (\sum_{k=1}^{n-1}
y_k^2)^2.
\ee
For the
logarithm of~\eqn{line2}, note that $|E^+|=|E^-|$ and so every $\phi$
which might appear is actually cancelled out. This gives
\bel{newterms} \sum_{j\in E^+} iy_j \,+
\sum_{j\in E^-} -iy_j 
\,+\sum_{j\in E}  \frac12 (y_j-y_{j+1})^2.
\ee
However, $y_n$ (if it appears) must be defined as $-\phi$, which is a
function of the other $y_j$. To handle the integral (of the exponential of
all these terms) we may use the averaging argument leading to~\cite[Equation
(2.4)]{MW}. This converts these terms into
$$ \sum_{j\in E^+} - y_j^2\, +
\sum_{j\in E}  \frac12 (y_j^2 +y_{j+1}^2) +o(1). 
$$
(The effect of the terms $iy_j$ can be gauged by considering the effect
of $J$ in the proof of~\cite[Lemma~3]{MW}. The terms involving $y_n$, if
they appear, similarly have negligible affect on the result of the
averaging. The only different type of term this might produce appears if
$n-1\in E$, so that $y_n^2$ appears. A term like the fourth power of $y_n$
appears in~\cite[Lemma~3]{MW} with coefficient $I$, and has negligible
effect.) To evaluate~\eqn{eye1} we must integrate the exponential of the sum
of this and~\eqn{old}. The result is easily seen to be 
$W\exp(-|E^+|/n+|E|/n) = W\exp(a(F_n)/n-s(F_n)/n)$ where
$W=(2\pi/n)^{(n-1)/2}\sqrt{n/e}$ is the original integral due to
the terms in~\eqn{old}. Note that together with the extra factors
outside the integral in~\eqn{adjustth} as compared to~\eqn{countth}, this
gives the result desired for the theorem.

It remains only to show that the part of the integral outside
$I_1$ is negligible. The factors $e^{i\th_j}$
in~\eqn{eye1} have absolute value 1 so can be ignored. For the rest, the
argument from~\cite{M} needs to be modified in a few small ways to be able
to cope with the extra factor $\prod_{j\in
E}\cos(\th_j-\th_{j+1})^{-1}$. We now refer heavily to the argument
in~\cite{M}, which splits the remaining part of the integral up into
pieces denoted $I_2(r)$ and $I_3(h)$. At the point where values
$\th_j$ satisfying
$|\th_j|<\pi/8$ are singled out for evaluating $I_2(r)$, one may instead
consider
$|\th_j|<1/\eta$  for a sufficiently small  
$\eta>0$. Then the value of $\prod_{j\in E}\cos(\th_j-\th_{j+1})^{-1}$ can
be made sufficiently small that it is dominated by the
exponentially small bound obtained on $I_2(r)$. For $I_3(h)$, a similar
modification is required: by  considering, in addition to what is already
in~\cite{M}, the number
$h'$ of $\th_j$ satisfying $n^{-1/2+\eta\epsilon}\le |\th_j|\le
n^{-1/2+\epsilon/4}$, this case is also seen to be negligible. The
lemma follows. \qed

\bigskip

\no
{\bf Proof of Theorem~\ref{t:timese}}
We first consider $\R(T^*)$.
If $T^*$ has only one vertex, then $\T(T^*)$ is just a random
tournament chosen from $S_n$, so $\ex [H(\T(T^*))]=g(n)$. By Lemma~\ref{paths}, the
probability that  a given Hamilton cycle occurs in a random regular
tournament is
$2^{-n}e(1+o(1))$. Since there are $(n-1)!$ possible Hamilton cycles, the
expected number occurring is $(n-1)!2^{-n}e(1+o(1))\sim eg(n)$, as
required.

If $T^*$ has more than one vertex, then by the first
statement in Lemma~\ref{final}, with probability tending to 1, a random
Hamilton cycle in $Q$ induces   
$(d-1)n/d+o(n)$ paths in any given $T_i$. These paths come in a linear
order, namely the order traversed in
$Q$ when it has a special initial vertex (see the proof of
Lemma~\ref{formulae}) and so they induce a forest $F_{i,n}$ as in
Lemma~\ref{paths}. Thus  
$a(F_{i,n})\sim n/d$. Given $a(F_{i,n})$, one may define $F_{i,n}$
by choosing a Hamilton path and specifying which of its arcs are to
be included. Random choice will choose two
consecutive edges asymptotically
$n/d^2+o(n)$ times (with probability tending to 1). Note that fixing
$F_{i,n}$,  the probability it appears in $T_i$ is exactly
$2^{-a(F_{i,n})}$. So with probability
tending to 1, $s(F_{i,n})\sim (d-1)n/d^2$.  For any set of forests
$F_{1,n},\ldots, F_{d,n}$ satisfying
$a(F_{i,n})\sim n/d$ and $s(F_{i,n})\sim (d-1)n/d^2$ for $1\le i\le d$,
Lemma~\ref{paths} implies that switching from random tournaments $T_i$ to
random regular tournaments $T_i$, which produces $\R(T^*)$, will
increase the probability of containment by 
$$
\prod_{i=1}^d e^{a(F_{i,n})/n-s(F_{i,n})/n} \sim  e^{n/d}.
$$
Thus $\ex [H(\R(T^*))]\sim e^{1/d} \ex[ H(\T(T^*))]$, as required.

For $\R_0(T^*)$, we may first modify the calculation in~\cite{AAR}
to see that the analogue of Lemma~\ref{paths} applies for the
tournaments considered there. This is rather more routine than the proof of
Lemma~\ref{paths} so we omit the details. The rest of the argument above
for $\R$ therefore applies to $\R_0$, and we obtain the desired result.
\qed

\bigskip

\no
{\bf Proof of Theorem~\ref{t:T3}\ }
We use the formula in Lemma~\ref{final}. As noted in the proof of
Lemma~\ref{formulae}, the coefficient in this formula counts walks in $T^*$,
which is in this case $C_3$. Thus,  the coefficient is  0
unless
$j_1=j_2=j_3=j$ for some $j$, in which case it is 1. Hence,
$$
\ex [H(\T(C_3))] \sim  \frac{3 (n!2^{-n})^3}{2n}\sum_{j\ge
1}    
   \left(\left(n-1\atop j-1\right) 2^{j}\right)^3.
$$
Routine analysis shows that the maximum contribution to the sum comes from
terms with
$j\sim 2n/3$ (as indicated in Lemma~\ref{final}) and, applying Stirling's
formula and expanding, one obtains
\bean
\ex [H(\T(T^*))] &\sim& \frac{3\sqrt 2}{n} \left(\frac{3n}{2e}\right)^n\sum_{x
 \in\ZZ} e^{-27x^2/4n}\\
&\sim& 
\sqrt{\frac{8\pi}{3n}} \left(\frac{3n}{2e}\right)^n\\
&\sim&2g(3n). \qed
\eean
 
\no
{\bf Proof of Theorem~\ref{t:forT}\ }
 Note that 
$$  \left(n-1\atop j-1\right) 2^{j}
=2^n[y^n](2y/(2-y))^j
$$
and thus by applying Lemma~\ref{final}
$$
\ex [H(\T(T^*))] \sim  f_n\frac{d ( n!)^d}{(d-1)n}\sim 
f_n\frac{d }{(d-1)n}( n/e)^{dn}(2\pi n)^{d/2}.
$$
Dividing this by $g(dn)$ and applying Stirling's formula gives the
expression for $\ex [H(\T(T^*))]$ stated in the theorem. To show that $\hat
f_n\sim f_n$, we may just note that the same formula results, by the same
argument, with $f(x_1,\ldots , x_n)$ defined to be any particular diagonal
element of the matrix $(I_d-AD)^{-1}$. \qed
\bigskip

To estimate $f_n$, we will apply 
the part of~\cite[Theorem 3.5]{PW} referring to the leading asymptotic term
of a ratio of analytic functions (in the case of non-zero numerator at the
singularity of interest).  This requires some preliminary setting.  Suppose
that  $G$ and $K$ are analytic functions of ${\bf y}=(y_1,\ldots ,y_d)$   in
a neigbourhood of  the toroidal region
$$U_\zv=\{{\bf y}\,:\, |y_j|\le |z_j|,\, 1\le j\le d\}$$
where ${\zv}=(z_1,\ldots, z_d)$,
  is the unique singular point of $G/K$ in $U$. We use the
suffix $j$ to denote partial differentiation with respect to $y_j$. Suppose
that  $K_d\ne 0$. Then (using the implicit function theorem) solving $K({\bf
y})=0$ yields  $y_d=h(y_1,\ldots ,y_{d-1})$ with $h$ analytic near $\zv$. As in
Section~1 before the statement of Theorem~\ref{t:coeff}, let
$\tilde f(\th_1,\ldots, \th_{d-1})=\log h (z_1e^{i\theta_1},\ldots,
z_{d-1}e^{i\theta_{d-1}})$, and let $\cal H$ be the determinant of
the Hessian matrix of $\tilde f$ at $(\th_1,\ldots
,\th_{d-1})=(0,\ldots , 0)$. Define
$\mbox{dir}({\zv})$ to be the set of complex scalar multiples  of
$(z_1K_1(\zv),
\ldots ,z_dK_d(\zv))$. (The reader may notice that $\tilde f$
differs from that in~\cite{PW} by a function which is linear in the
$\theta_i$. The definitions are
equivalent for the current purpose because 
$\tilde f$ enters the following formula only through $\HH$.)

\begin{thm}\cite{PW} \lab{t:PW}
With the assumptions and definitions in the paragraph above, if $G(\zv)\ne 0$ and
$\HH\ne 0$ then for
$\rv\in\mbox{dir}(\zv)$
$$[y_1^{r_1}\cdots y_d^{r_d}]\,\frac{G(\yv)}{K(\yv)}
\sim
(2\pi)^{(1-d)/2}\frac{G(\zv)}{-z_dK_d(\zv)\sqrt{\HH}}\prod_{j=1}^dz_j^{-r_j}.$$
\end{thm}

\no
{\bf Proof of Theorem~\ref{t:coeff}\ }  Define $\tilde B=I_d-AD$. Write
$C_{[kl]} $ for the result of deleting  row $k$ and column $l$ from a matrix
$C$, and let $\phi$ denote the formal power series homomorphism generated by
$$\phi(x_j)=\frac{y_j}{1-\frac12 y_j }$$
for all $j$. Extend $\phi$ in the natural way to act on matrices by acting on all
elements of the matrix. Since 
$\tilde B^{-1}=(\det
\tilde B)^{-1}\mbox{adj }\tilde B$ and $f(x_1,\ldots, x_d)$ is the  $(1,1)$ entry of
$\tilde B^{-1}$, we have
 $$
 f(x_1,\ldots, x_d)  =  \frac{\det \tilde B_{[11]}}{\det \tilde B}.
$$
 Applying $\phi$ to both sides,
\bean
 \phi(f(x_1,\ldots, x_d))  &=&  \frac{\det \phi(\tilde B_{[11]})}{\det \phi(\tilde
B)}\\ &=&  \frac{(1-\frac12 y_1 )\det   B_{[11]} }{\det   B}
\eean
where the last step uses the fact that 
\bel{Brel}\textstyle
B=\phi(\tilde B)\, \mbox{diag}(1-\frac12 y_1,\ldots,1-\frac12 y_d). 
\ee
It thus follows from~\eqn{fndef} that
\bel{fasy}
f_n =[y_1^n\cdots y_d^n] \frac{(1- \frac12 y_1 )\det  B_{[11]}  }{ \det
B}.
\ee

To apply this result, we must find a suitable singular point $\zv$  of the function
in~\eqn{fasy}, so that
$\mbox{dir}({\zv})$ contains $(n,n,\ldots ,n)$. First note that
$\det  B_{[11]}$ and $\det  B$ are both polynomials, and hence
entire functions.  Moreover, since  $T^*$ is regular, all the row sums of $ B$ are
  equal to 0  when
$(y_1,\ldots, y_d)={\zv}$ where $\zv$ was
defined in~\eqn{z}. Thus  $\det   B(\zv)=0$, which makes $\zv$ a singularity of
the rational function in~\eqn{fasy} unless the numerator is also 0 at $\zv$. 

Note that $\frac{d}{2}B({\bf z})$ is the
Laplacian matrix of
$T^*$, and so by the Matrix-Tree Theorem~\cite{T}, 
$\det\frac{d}{2}B_{[11]}({\bf z})$ is the number of directed spanning
trees in
$T^*$ rooted at vertex 1. Since $T^*$ is a regular tournament, it is
strongly connected. (This follows because, for any partition of the vertex set
into two nonempty parts $U$ and $V$, there must be an equal number of edges
directed from $U$ to $V$ as from $V$ to $U$.) Hence the number of directed
spanning trees with any given root vertex is non-zero. Thus the numerator
in~\eqn{fasy} is non-zero at
$\bf z$, and so $\zv$ is indeed a singularity. 

To show that there are no other singularities in an appropriate toroidal region,
it suffices to show that
$\det  B\ne 0$ at an arbitrary point $\yv = (y_1,\ldots, y_d)$ where $ |y_j|
\le  z$ for each $j$, and
$|y_k|<z_k$ for some $k$. To this end, let $ 
D=\mbox{diag}( y_1 ,\ldots, y_d )$, $ 
\hat D=\mbox{diag}( |y_1| ,\ldots, |y_d| )$, $D_z=\mbox{diag}(z,\ldots,z)$ and
$\hat A=\frac12 I_d+A$. Then $B=I_d-\hat A D$. Assume that $\det  B= 0$. It
follows that 1 is an eigenvalue of $\hat A D$. Let $\vv$ be an associated
eigenvector. Then 
$$ |\hat A   D \vv|= |\vv|$$ 
where $|C|$ denotes the matrix
whose entries are the absolute values of the entries in $C$. Hence  
$$ \hat A \hat D|\vv|\ge |\vv|,$$
where $\ge$ denotes elementwise comparison. That is, $\vv$ is a 1-subharmonic
vector for $ \hat A \hat D$. So by~\cite[Lemma 8.7.1]{GR} the largest
eigenvalue of $ \hat A \hat D$ is at least 1. Note that
     $\hat AD_z$, $\hat
A\hat D$ and
$\hat AD_z-\hat A\hat D$ are all nonegative real matrices, $\hat AD_z\ne \hat
A\hat D$, and furthermore the underlying directed graph of $\hat A$ consists of
$T^*$ with a loop at each vertex, which is strongly connected. Thus by the
Perron-Frobenius Theorem (see~\cite[Theorem 8.8.1]{GR} in particular), the
maximum eigenvalue of $\hat A\hat D$ is strictly less than that of $\hat
AD_z$, which must therefore be strictly greater than 1. On the other hand,
it is exactly  1 because all row sums of
$\hat A$ are $d/2=1/z$. This is a contradiction, so $\det  B\ne 0$ at $\yv$.

Furthermore, with
$K({\bf y})=\det B$, we need to find $K_j({\bf z})$ for each $j$.
Note that $K_j(\zv)=\det\frac{\partial}{\partial y_j}B(\zv)$. But  
$\frac{\partial}{\partial y_j}B(\zv)$ and $B(\zv)$ only differ in
column
$j$, which in the former has $-1$ in the off-diagonal entries where the latter has
$-2/d$, and
$-\frac12$ in the diagonal entry where the latter has $1-1/d$.  Thus, expanding the
determinants down column $j$ shows that   
\bea
K_j({\bf z})&=& \frac{d}{2} \det B({\bf z})-\frac{d}{2}\det  B_{[jj]}({\bf z})\non\\
&=&-\frac{d}{2} \det  B_{[jj]}({\bf z}).\lab{kj}
\eea
Since $T^*$ is regular it is balanced, and so by the BEST
theorem (see~\cite{GJ} for example), the number of directed spanning trees
rooted at vertex
$j$ is independent of
$j$ (this independence can alternatively be derived from $A\cdot{\rm
adj}A={\bf 0}$), and is equal to
$\frac{d-1}{2}\det B_{[jj]}({\bf z})$. Thus
$K_j({\bf z})$ is independent of $j$, in particular $K_d(\zv)\ne 0$ which
completes verification of the hypotheses of Theorem~\ref{t:PW}. This independence of
$j$ also shows that
${\bf dir}({\bf z})$ contains the vector $(1,1,\ldots , 1)$ as desired.
By~\eqn{kj}, it also  shows that
$-zK_d(\zv)$ cancels with $\det B_{[11]}({\bf z})$ in the numerator (here $G(\zv)= (1-
z/2)\det  B_{[11]}({\bf z})$). 
Thus by  Theorem~\ref{t:PW} 
\bel{fasyfirst}
f_n\sim \left(\frac{d}{2}\right)^{ dn } (2\pi n)^{-(d-1)/2}
\frac{(1- z/2)  }{  \sqrt\HH}.
\ee 
The theorem  now
follows from Theorem~\ref{t:forT} on recalling that $z=2/d$.
 \qed
\bigskip

\no
{\bf Proof of Theorem~\ref{t:paths}\ }
 We modify the treatment of the random variable $H$, which counts
cycles, for
$H'$, which counts paths.  Define
$f^{(kl)}(x_1,\ldots ,x_n)$ to be the $(k,l)$ entry of $(I_d-AD)^{-1}$.
Then  the argument in Lemma~\ref{formulae} applies for Hamilton paths
which begin at a vertex of $T_k$ and end at $T_l$, with the
modifications that $W$ is no longer necessarily closed, there is no
overcounting by
$j_1$, and there is an initial path in $T_k$ corresponding to the first
vertex of $W$. This applies even if $W$ happens to be closed; the
initial and final vertices of $W$ correspond to paths which are
essentially independent. So we have  
$$
\ex [H'(\R(T^*))]  = (n!2^{-n})^d
\sum_{1\le k,l\le d}
\sum_{j_1,\ldots , j_d\ge
1}  [x_1^{j_1}\cdots x_d^{j_d}] x_kf^{(kl)}(x_1,\ldots ,x_d)
\prod_{i=1}^d  \left(n-1\atop j_i-1\right) 2^{j_i}. 
$$
Thus  (c.f.\ Theorem~\ref{t:forT} and~\eqn{fasy} )
$$
\ex [H'(\R(T^*))]\sim
g(dn)\left( \frac{2}{d}\right)^{dn}(2\pi
n)^{(d-1)/2}d^{1/2}n\sum_{1\le k,l\le d} f^{(kl)}_n 
$$
where
\bean
f^{(kl)}_n &=& [y_1^n\cdots y_d^n] \frac{y_k}{1-\frac12 y_k}
f{\textstyle \left(1-\frac12 y_1,\ldots,1-\frac12 y_d\right)}\\
 &=&[y_1^n\cdots y_d^n] \frac{y_k\det  B_{[kl]}  }{ \det
B}.
\eean
The Matrix-Tree Theorem (or, again, linear algebra) implies that $\det 
B_{[kl]}=\det  B_{[kk]}$, which, as already noted in the proof of
Theorem~\ref{t:coeff}, is equal to
$B_{[11]}$. Hence, applying the proof of Theorem~\ref{t:coeff}, we
may replace $y_k$ by $z$ and we find
$f^{(kl)}_n\sim z f_n/(1-\frac12 z)$. So in summary,
$\ex [H'(\R(T^*))]$ (in comparison with
$\ex [H(\R(T^*))]$) gains factors of $n(d-1)/d$ (from  the missing  $j_1^{-1}$)
and
$z/(1-\frac12 z)=2/(d-1)$, and the factor of $d^2$ for summing
over
$k$ and $l$. Since $g'(dn)/g(dn)\sim 2dn$, the claimed path analogue
of Theorem~\ref{t:coeff} follows.

 For the path analogue of
Theorem~\ref{t:timese}, the proof goes through with no more changes,
using Lemma~\ref{paths}.
\qed 

\end{document}